\documentclass[11pt,reqno]{amsart}

\usepackage{graphicx}

\usepackage{xspace}
\usepackage{amsmath,amsthm,amsfonts,amssymb,latexsym,mathrsfs,color}
\usepackage{hyperref}

\newtheorem{theorem}{Theorem}
\newtheorem{corollary}[theorem]{Corollary}
\newtheorem{proposition}[theorem]{Proposition}

\newtheorem{definition}[theorem]{Definition}

\newcommand{\negg}{{\rm neg\,}}

\newcommand{\altrun}{{\rm altrun\,}}

\newcommand{\udrun}{{\rm udrun\,}}
\newcommand{\val}{{\rm val\,}}
\newcommand{\fap}{{\rm fap\,}}

\newcommand{\fdes}{{\rm fdes\,}}

\newcommand{\ap}{{\rm ap\,}}
\newcommand{\lap}{{\rm lap\,}}
\newcommand{\lpk}{{\rm lpk\,}}

\newcommand{\ipk}{{\rm ipk\,}}
\newcommand{\pk}{{\rm pk\,}}

\newcommand{\des}{{\rm des\,}}

\newcommand{\exc}{{\rm exc\,}}
\newcommand{\fexc}{{\rm fexc\,}}

\newcommand{\mdn}{\mathcal{S}^D_n}

\newcommand{\msn}{\mathcal{S}_n}

\newcommand{\lrf}[1]{\lfloor #1\rfloor}

\newcommand{\mbn}{{\mathcal S}^B_n}

\newcommand{\mqn}{\mathcal{Q}_n}

\DeclareMathOperator{\LHP}{LHP}
\DeclareMathOperator{\R}{\mathbb{R}}
\DeclareMathOperator{\re}{Re}

\newcommand{\eulerian}[2]{\genfrac{|}{|}{0pt}{}{#1}{#2}}

\title{Eulerian pairs and Eulerian recurrence systems}
\author[S.-M.~Ma]{Shi-Mei Ma}
\address{School of Mathematics and Statistics,
        Northeastern University at Qinhuangdao,
         Hebei 066000, P.R. China}
\email{shimeimapapers@163.com (S.-M. Ma)}
\author[J.~Ma]{Jun Ma}
\address{Department of mathematics, Shanghai Jiao Tong University, Shanghai, P.R. China}
\email{majun904@sjtu.edu.cn(J.~Ma)}
\author{Jean Yeh}
\address{Department of Mathematics, National Kaohsiung Normal University, Kaohsiung 82444, Taiwan}
\email{chunchenyeh@nknu.edu.tw}
\author[Y.-N. Yeh]{Yeong-Nan Yeh}
\address{Institute of Mathematics,
        Academia Sinica, Taipei, Taiwan}
\email{mayeh@math.sinica.edu.tw (Y.-N. Yeh)}
\thanks{The first and third authors are both corresponding authors}
\begin{document}

\maketitle
\begin{abstract}
In this paper, we characterize a duality relation between Eulerian recurrences and Eulerian recurrence systems,
which generalizes and unifies Hermite-Biehler decompositions of several enumerative polynomials, including flag descent polynomials for hyperoctahedral group, flag ascent-plateau polynomials for Stirling permutations, up-down run polynomials for symmetric group and alternating run polynomials for hyperoctahedral group.
As applications, we derive some properties of associated enumerative polynomials.
In particular, we find that both the ascent-plateau polynomials and left ascent-plateau polynomials for Stirling permutations are alternatingly increasing,
and so they are unimodal with modes in the middle.

\bigskip

\end{abstract}
\date{\today}
\section{Introduction}
Given $m,n\in \mathbb{N}$. When $m\leq n$, let $[m,n]=\{m,m+1,\ldots,n\}$. As usual, let $[n]=\{1,2,\ldots,n\}$.
The cardinality of a set $A$ will be
denoted by $\#A$. Let $\msn$ be the set of all permutations of $[n]$ and
let $\pi=\pi(1)\pi(2)\cdots\pi(n)\in\msn$.
Let $\mbn$ be the hyperoctahedral group of rank $n$. Elements of $\mbn$ are signed permutations $\pi$ of the set $\pm[n]$ such that $\pi(-i)=-\pi(i)$ for all $i$, where $\pm[n]=\{\pm1,\pm2,\ldots,\pm n\}$.
Let $\mdn$ denote the group of even signed permutations, which is a Coxeter group of type $D$ of rank $n$.
The group $\mdn$ is the subgroup of $\mbn$ consisting of signed permutations with an even number of negative entries among $\pi(1),\pi(2)\ldots,\pi(n)$.
Define
\begin{equation*}
\begin{split}
\des_A(\pi)&:=\#\{i\in [n-1]:~\pi(i)>\pi({i+1})\},\\
\des_B(\pi)&:=\#\{i\in [0,n-1]:~\pi(i)>\pi({i+1}),~\pi(0)=0\},\\
\des_D(\pi)&:=\#\{i\in[0,n-1]:~\pi(i)>\pi({i+1}),~\pi(0)=-\pi(2)\}.
\end{split}
\end{equation*}
The types $A,B$ and $D$ Eulerian polynomials are
respectively defined by
\begin{equation*}
\begin{split}
A_n(x)&=\sum_{\pi\in\msn}x^{\des_A(\pi)},~
B_n(x)=\sum_{\pi\in\mbn}x^{\des_B(\pi)},~
D_n(x)=\sum_{\pi\in \mdn}x^{\des_D(\pi)}.
\end{split}
\end{equation*}
There is a close connection among the three types of Eulerian polynomials~\cite[Lemma 9.1]{Stembridge94}:
\begin{equation}\label{Dnx-BnxAnx}
D_n(x)=B_n(x)-n2^{n-1}xA_{n-1}(x) \quad {\text for}\quad n\geq 2.
\end{equation}
Subsequently, Brenti~\cite[Theorem~4.7]{Brenti94} obtained a $q$-analogue of~\eqref{Dnx-BnxAnx}. There are several recurrences for $D_n(x)$ (see~\cite{Chow03,Hyatt16}).
The polynomials $A_n(x),B_n(x)$ and $D_n(x)$ have several common properties, including unimodality, real-rootedness and $\gamma$-positivity
(see~\cite{Athanasiadis17,Hyatt16,Liu07,Petersen07,Savage15,Yang}).
This paper is motivated by the recent work of Hyatt~\cite{Hyatt16} and Hwang, Chern and Duh~\cite{Hwang20}.

Let $\mbn=B_n^+\cup B_n^-$ and $\mdn=D_n^+\cup D_n^-$, where
$$B_n^+=\{\pi\in\mbn: \pi(n)>0\},~B_n^-=\{\pi\in\mbn:~\pi(n)<0\},$$
$$D_n^+=\{\pi\in\mdn: \pi(n)>0\},~D_n^-=\{\pi\in\mdn:~\pi(n)<0\}.$$
Define $$P_n(x)=\sum_{k=0}^{n-1}\binom{n}{k}B_k(x)(x-1)^{n-k-1},~Q_n(x)=\sum_{k=0}^{n-1}\binom{n}{k}D_k(x)(x-1)^{n-k-1}.$$
Given a polynomial $f(x)$ of degree $n$. Let $\widetilde{f}(x)=x^nf(1/x)$.
In~\cite{Hyatt16}, Hyatt found that
\begin{equation}\label{Bnx}
B_n(x)=P_n(x)+x^nP_n(1/x)=P_n(x)+x\widetilde{P}_n(x)~\text{for $n\geq 1$},
\end{equation}
\begin{equation*}\label{Dnx}
D_n(x)=Q_n(x)+x^nQ_n(1/x)=Q_n(x)+x\widetilde{Q}_n(x)~\text{for $n\geq 2$},
\end{equation*}
where
\begin{equation*}
P_n(x)=\sum_{\pi\in B_n^+}x^{\des_B(\pi)},~x^nP_n(1/x)=\sum_{\pi\in B_n^-}x^{\des_B(\pi)},
\end{equation*}
\begin{equation*}
Q_n(x)=\sum_{\pi\in D_n^+}x^{\des_D(\pi)},~x^nQ_n(1/x)=\sum_{\pi\in D_n^-}x^{\des_D(\pi)}.
\end{equation*}
Motivated by~\eqref{Bnx}, in the following we shall present a more natural decomposition of $B_n(x)$.

Following~\cite{Adin2001}, the {\it flag descent number} of $\pi\in\mbn$
is defined by
\begin{equation*}
\fdes(\pi):=\begin{cases}
2\des_A(\pi)+1,& \text{if $\pi(1)<0$};\\
2\des_A(\pi), & \text{otherwise}.
\end{cases}
\end{equation*}
Clearly, $\fdes(\pi)=\des_A(\pi)+\des_B(\pi)$. The {\it flag descent polynomial} is defined by
$$C_n(x)=\sum_{\pi\in \mbn}x^{\fdes(\pi)}.$$
It follows from~\cite[Theorem~4.4]{Adin2001} that
\begin{equation}\label{fdesAnx}
C_n(x)=(1+x)^nA_n(x),
\end{equation}
and so $C_n(x)$ is symmetric and unimodal. Let $\negg(\pi):=\#\{i\in [n]:~\pi(i)<0\}$.
Consider the $q$-flag descent polynomials
$$C_n(x,q)=\sum_{\pi\in \mbn}x^{\fdes(\pi)}q^{\negg(\pi)}.$$
Set $\mbn=C_n^+\cup C_n^-$, where $C_n^+=\{\pi\in\mbn:~\pi(1)>0\}$ and $C_n^-=\{\pi\in \mbn:~\pi(1)<0\}$.
Define
$$C_n^E(x,q)=\sum_{\pi\in C_n^+}x^{\des_A(\pi)}q^{\negg(\pi)},~C_n^O(x,q)=\sum_{\pi\in C_n^-}x^{\des_A(\pi)}q^{\negg(\pi)}.$$
Then we have
\begin{equation}\label{SnxqEO}
C_n(x,q)=C_n^E(x^2,q)+xC_n^O(x^2,q).
\end{equation}

Consider the {\it type $B$ $q$-Eulerian polynomials}
$$B_n(x,q)=\sum_{\pi\in \mbn}x^{\des_B(\pi)}q^{\negg(\pi)}.$$
The polynomials $B_n(x,q)$ satisfy the recurrence
$$B_{n+1}(x,q)=((1+q)nx+qx+1)B_n(x,q)+(1+q)x(1-x)B_n'(x,q),$$
with $B_0(x,q)=1$ (see~\cite[Theorem~3.4]{Brenti94}).
Note that $\des_B(\pi)=\des_A(\pi)$ for $\pi\in C_n^+$ and $\des_B(\pi)=\des_A(\pi)+1$ for $\pi\in C_n^-$.
Hence
\begin{equation*}
B_n(x,q)=C_n^E(x,q)+xC_n^O(x,q).
\end{equation*}
Set $C_n^E(x,1)=C_n^E(x)$ and $C_n^O(x,1)=C_n^O(x)$. By~\eqref{fdesAnx}, the polynomial $C_n(x)$ is symmetric. Thus $C_n^E(x)=x^{n-1}C_n^O(1/x)$ for $n\geq 1$.
Note that $B_n(x,1)=B_n(x)$ and $C_n(x,1)=C_n(x)$. We can now conclude the following result from the discussion above.
\begin{proposition}\label{SnxBnx-prop}
For $n\geq 1$, we have
\begin{equation}\label{SnxBnx-decom}
C_n(x)=C_n^E(x^2)+xC_n^O(x^2),~B_n(x)=C_n^E(x)+xC_n^O(x).
\end{equation}
\end{proposition}

Comparing~\eqref{Bnx} with~\eqref{SnxBnx-decom}, we get the following result, and we give a proof of it for completeness.
\begin{proposition}
We have $C_n^E(x)=P_n(x)$ and $C_n^O(x)=\widetilde{P}_n(x)$.
\end{proposition}
\begin{proof}
Define
\begin{align*}
{^+B_n^+}&=\{\pi\in B_n: \pi(1)>0,~ \pi(n)>0\},\\
{^+B_n^-}&=\{\pi\in B_n: \pi(1)>0,~ \pi(n)<0\},\\
{^-B_n^+}&=\{\pi\in B_n: \pi(1)<0,~ \pi(n)>0\},\\
{^-B_n^-}&=\{\pi\in B_n: \pi(1)<0,~ \pi(n)<0\}.
\end{align*}
Note that $C_n^+={^+B_n^+}\cup {^+B_n^-}$ and $B_n^+={^+B_n^+}\cup {^-B_n^+}$.
A bijection $\Phi$ from $C_n^+$ to $B_n^+$ is given as follows:
\begin{itemize}
  \item [$(i)$] If $\pi\in {^+B_n^+}$, then let $\Phi(\pi)=\pi$;
  \item  [$(ii)$] For $\pi\in {^+B_n^-}$, let $k$ be the smallest index of $\pi$ such that $\pi(k)>0$ and $\pi(k+1)<0$.
  Then we define $\Phi(\pi)=\pi(k+1)\cdots \pi(n)\pi(1)\cdots\pi(k)$.
\end{itemize}
Note that $\des_B\left(\Phi(\pi)\right)=\des_B(\pi)$. Hence
\begin{equation*}
\sum_{\pi\in C_n^+}x^{\des_B(\pi)}=\sum_{\pi\in B_n^+}x^{\des_B(\pi)},
\end{equation*}
Along the same lines, it is easy to show that
$\sum_{\pi\in C_n^-}x^{\des_B(\pi)}=\sum_{\pi\in B_n^-}x^{\des_B(\pi)}$.
\end{proof}

Motivated by Proposition~\ref{SnxBnx-prop}, in the following we shall establish a duality relation between Eulerian recurrences and Eulerian recurrence systems.
It is well known that the Eulerian polynomials $A_n(x)$ and $B_n(x)$ satisfy the following recurrence relations:
\begin{equation*}
\begin{split}
A_{n}(x)&=(nx+1-x)A_{n-1}(x)+x(1-x)A_{n-1}'(x),\\
B_{n}(x)&=(2nx+1-x)B_{n-1}(x)+2x(1-x)B_{n-1}'(x),
\end{split}
\end{equation*}
with the initial conditions $A_0(x)=B_0(x)=1$ (see~\cite{Brenti94,Zhuang17}).
In recent years, there has been much work on the generalizations of Eulerian recurrences, see~\cite{Barbero14,Hwang20,Zhu} and references therein.
In~\cite{Barbero14}, Barbero, Salas and Villase\~{n}or systematically studied and classified the partial differential equations that are satisfied by
the generating function
$$f(x,y)=\sum_{n,k\geq 0}\eulerian{n}{k}x^k\frac{y^n}{n!},$$
where the numbers $\eulerian{n}{k}$ satisfy the recurrence relation
$$\eulerian{n}{k}=(\alpha n+\beta k+\gamma)\eulerian{n-1}{k-1}+(\alpha' n+\beta' k+\gamma')\eulerian{n-1}{k},$$
with $\eulerian{0}{0}=1$ and $\eulerian{n}{k}=0$ when $n<0$ or $k<0$.
Very recently, Hwang, Chern and Duh~\cite{Hwang20} considered the general Eulerian recurrence:
\begin{equation}\label{Eulerian02}
P_{n}(x)=(\alpha(x)n+\gamma(x))P_{n-1}(x)+\beta(x)(1-x)P_{n-1}'(x),
\end{equation}
with $P_0(x),\alpha(x),\beta(x)$ and $\gamma(v)$ are given functions of $x$ (they are often polynomials). They studied the limiting distribution of the coefficients of $P_n(x)$ for large $n$ when the coefficients are nonnegative. In particular, Hwang, Chern and Duh~\cite[Section~9.3]{Hwang20} discussed the limiting distribution of the coefficients polynomials that satisfy Eulerian recurrence systems.

\begin{definition}\label{def01}
Let $\{E_n(x)\}_{n\geq 0}$ and $\{O_n(x)\}_{n\geq 0}$ be two sequences of polynomials.
We say that the ordered pair of polynomials $(E_n(x),O_n(x))$ is a Eulerian pair if $\deg E_n(x)\geq \deg O_n(x)$ and
the polynomials $E_n(x)$ and $O_n(x)$ satisfy the Eulerian recurrence system:
\begin{equation}\label{EulerianPair}
\left\{
  \begin{array}{l}
    E_{n+1}(x)=p_n(x)E_n(x)+q_n(x)E_n'(x)+r_n(x)O_n(x), \\
    O_{n+1}(x)=u_n(x)O_n(x)+v_n(x)O_n'(x)+w_n(x)E_n(x),
  \end{array}
\right.
\end{equation}
where $E_0(x),O_0(x),p_n(x),q_n(x),r_n(x),u_n(x),v_n(x),w_n(x)$ are given polynomials of low degrees.
\end{definition}

Let $f(x)=\sum_{i=0}^nf_ix^i$. Throughout this paper, we always let
\begin{align*}
f^E(x)&=\sum_{k=0}^{\lrf{\frac{n}{2}}}f_{2k}x^{k},~f^O(x)=\sum_{k=0}^{\lrf{\frac{n-1}{2}}}f_{2k+1}x^{k};\\
f^e(x)&=\sum_{k=0}^{\lrf{\frac{n}{2}}}f_{2k}x^{2k},~f^o(x)=\sum_{k=0}^{\lrf{\frac{n-1}{2}}}f_{2k+1}x^{2k+1}.
\end{align*}
Hence $f(x)=f^E(x^2)+xf^O(x^2)=f^e(x)+f^o(x)$.
As an extension of~\eqref{Eulerian02}, we define a sequence of polynomials $F_n(x)$ by using the following general Eulerian recurrence:
\begin{equation}\label{Eulerian03}
F_{n+1}(x)=\alpha_n(x)F_{n}(x)+\beta_n(x)F_{n}'(x),
\end{equation}
where $F_0(x),\alpha_n(x)$ and $\beta_n(x)$ are given polynomials of low degrees.
We can now present the first main result of this paper.
\begin{theorem}\label{mathm01}
Let $\langle E_n(x),O_n(x)\rangle$ be a Eulerian pair that satisfies the Eulerian recurrence syetem~\eqref{EulerianPair},
and let $F_n(x)$ be the polynomial defined by the recurrence~\eqref{Eulerian03}.
Then the polynomial $F_n(x)$ has the expression $F_n(x)=E_n(x^2)+xO_n(x^2)$ if and only if the following conditions hold:
\begin{align*}
u_n(x)&=p_n(x)+\frac{1}{2x}q_n(x),~v_n(x)=q_n(x),~w_n(x)=\frac{1}{x}r_n(x),\\
\alpha_n(x)&=p_n(x^2)+\frac{1}{x}r_n(x^2),~\beta_n(x)=\frac{1}{2x}q_n(x^2),~\beta_n^e(x)=0.
\end{align*}
\end{theorem}
\begin{proof}
By using $F_n(x)=E_n(x^2)+xO_n(x^2)$, we obtain
\begin{equation*}
F_n'(x)=2xE_n'(x^2)+O_n(x^2)+2x^2O_n'(x^2).
\end{equation*}
Then it follows from~\eqref{Eulerian03} that
\begin{align*}
F_{n+1}(x)&=\alpha_n(x)\left(E_n(x^2)+xO_n(x^2)\right)+\\
&\beta_n(x)\left(2xE_n'(x^2)+O_n(x^2)+2x^2O_n'(x^2)\right).
\end{align*}
Comparing this with the expression
$F_{n+1}(x)=E_{n+1}(x^2)+xO_{n+1}(x^2)$, we obtain
\begin{align*}
E_{n+1}(x^2)&=\alpha_n^e(x)E_n(x^2)+x\alpha_n^o(x)O_n(x^2)+\\
&\beta_n^e(x)\left(O_n(x^2)+2x^2O_n'(x^2)\right)+2x\beta_n^o(x)E_n'(x^2),\\
O_{n+1}(x^2)&=\frac{1}{x}\alpha_n^o(x)E_n(x^2)+\alpha_n^e(x)O_n(x^2)+\\
&\frac{1}{x}\beta_n^o(x)\left(O_n(x^2)+2x^2O_n'(x^2)\right)+2\beta_n^e(x)E_n'(x^2).
\end{align*}
Since~\eqref{EulerianPair} holds, then $\beta_n^e(x)=0$.
Hence
\begin{equation}\label{EulerianPair02}
\left\{
  \begin{array}{l}
E_{n+1}(x^2)=\alpha_n^e(x)E_n(x^2)+2x\beta_n^o(x)E_n'(x^2)+x\alpha_n^o(x)O_n(x^2),\\
O_{n+1}(x^2)=\left(\alpha_n^e(x)+\frac{1}{x}\beta_n^o(x)\right)O_n(x^2)+2x\beta_n^o(x)O_n'(x^2)+\frac{1}{x}\alpha_n^o(x)E_n(x^2).
  \end{array}
\right.
\end{equation}
By comparing~\eqref{EulerianPair} with~\eqref{EulerianPair02}, we immediately get the following relations:
\begin{align*}
p_n(x^2)&=\alpha_n^e(x),~q_n(x^2)=2x\beta_n^o(x),~r_n(x^2)=x\alpha_n^o(x),\\
u_n(x^2)&=\alpha_n^e(x)+\frac{1}{x}\beta_n^o(x),~v_n(x^2)=2x\beta_n^o(x),~w_n(x^2)=\frac{1}{x}\alpha_n^o(x),
\end{align*}
which yield the desired result. This completes the proof.
\end{proof}

Let $F_n(x)$ be the polynomials defined by the recurrence~\eqref{Eulerian03}.
When $\beta_n^e(x)=0$, we can get the recurrence system~\ref{EulerianPair} by using
the following relations:
\begin{equation}\label{dualrelation01}
x\alpha_n(x)=xp_n(x^2)+r_n(x^2),~2x\beta_n(x)=q_n(x^2).
\end{equation}

A polynomial $p\in\R[x]$ is ({\it Hurwitz} or {\it asymptotically})
{\it stable} if every zero of $p$ is in the open left half plane
$\LHP=\{z\colon\re z<0\}$;
$p$ is {\it standard} if its leading coefficient is positive.
Suppose that $p,q\in\R[x]$ both have only real zeros,
that those of $p$ are $\xi_1\leqslant\cdots\leqslant\xi_n$,
and that those of $q$ are $\theta_1\leqslant\cdots\leqslant\theta_m$.
We say that $p$ {\it interlaces} $q$ if $\deg q=1+\deg p$ and the zeros of
$p$ and $q$ satisfy
$$
\theta_1\leqslant\xi_1\leqslant\theta_2\leqslant\cdots\leqslant\xi_n
\leqslant\theta_{n+1}.
$$
We also say that $p$ {\it alternates left of} $q$ if $\deg p=\deg q$
and the zeros of $p$ and $q$ satisfy
$$
\xi_1\leqslant\theta_1\leqslant\xi_2\leqslant\cdots\leqslant\xi_n
\leqslant\theta_n.
$$
We use the notation $p\prec q$ for
either ``$p$ interlaces $q$'' or ``$p$ alternates left of $q$''.
The next theorem is a version of the classical Hermite-Biehler theorem.
\begin{theorem}[{\cite[Theorem~4.1]{Branden11}}]\label{Hermite}
Let $F(x)=F^E(x^2)+xF^O(x^2)\in\R[x]$ be standard.
Then $F(x)$ is stable if and only if both $F^E(x)$ and $F^O(x)$ are standard,
have only nonpositive zeros, and $F^O(x)\prec F^E(x)$.
\end{theorem}

The Hermite-Biehler theorem has been widely used to study the distribution of zeros. In this paper, we consider combinatorial aspects of the expression $F(x)=F^E(x^2)+xF^O(x^2)$, which we call the {\it Hermite-Biehler decomposition} of $F(x)$.
In Section~\ref{Section02}, we collect the Hermite-Biehler decompositions of several enumerative polynomials,
including flag descent polynomials for hyperoctahedral group, flag ascent-plateau polynomials for Stirling permutations, up-down run polynomials for symmetric group and
alternating run polynomials for hyperoctahedral group. As applications of Theorem~\ref{mathm01}, we get some new properties of associated polynomials.
\section{Eulerian recurrence systems and Hermite-Biehler decompositions}\label{Section02}
\subsection{The flag descent polynomials}
\hspace*{\parindent}

It follows from~\cite[Theorem 4.3]{Adin2001} that
the flag descent polynomials $C_n(x)$ satisfy the recurrence
\begin{equation}\label{AdinSnx-recu}
C_{n+1}(x)=(2nx^2+x+1)C_n(x)+x(1-x^2)C_n'(x),
\end{equation}
with $C_0(x)=1,C_1(x)=1+x$ and $C_2(x)=1+3x+3x^2+x^3$.
Set $$\alpha_n(x)=2nx^2+x+1,~\beta_n(x)=x(1-x^2).$$ Note that $\beta_n^e(x)=0$.
It follows from~\eqref{dualrelation01} that
$$x(2nx^2+x+1)=xp_n(x^2)+r_n(x^2),~2x^2(1-x^2)=q_n(x^2).$$
Hence $p_n(x)=2nx+1,~q_n(x)=2x(1-x),~r_n(x)=x$.
By Theorem~\ref{mathm01}, we see that
$$u_n(x)=2nx+1+1-x=(2n-1)x+2,~v_n(x)=2x(1-x),~w_n(x)=1.$$
By Theorem~\ref{mathm01} and Theorem~\ref{Hermite}, we get the second main result of this paper.
\begin{theorem}\label{Snx-thm}
For $n\geq1$, we have
\begin{equation}\label{Snx-recu}
\left\{
  \begin{array}{l}
    C_{n+1}^E(x)=(2nx+1)C_n^E(x)+2x(1-x)\frac{d}{dx}{C_n^E}(x)+xC_n^O(x), \\
    C_{n+1}^O(x)=((2n-1)x+2)C_n^O(x)+2x(1-x)\frac{d}{dx}C_n^O(x)+C_n^E(x),
  \end{array}
\right.
\end{equation}
with the initial conditions $C_1^E(x)=C_1^O(x)=1$. Moreover, both $C_n^E(x)$ and $C_n^O(x)$ have only real negative zeros and $C_n^O(x)\prec C_n^E(x)$.
\end{theorem}

Define
$$C^E(x,z)=1+\sum_{n=1}^\infty C_n^E(x)\frac{z^n}{n!},~C^O(x,z)=\sum_{n=1}^\infty C_n^O(x)\frac{z^n}{n!}.$$
By rewriting~\eqref{Snx-recu} in terms of generating functions, we have
\begin{equation}\label{Snx-recu-GF}
\left\{
  \begin{array}{l}
    (1-2xz)\frac{\partial}{\partial z}C^E(x,z)=C^E(x,z)+2x(1-x)\frac{\partial}{\partial x}C^E(x,z)+xC^O(x,z),\\
      (1-2xz)\frac{\partial}{\partial z}C^O(x,z)=(2-x)C^O(x,z)+2x(1-x)\frac{\partial}{\partial x}C^O(x,z)+C^E(x,z).
  \end{array}
\right.
\end{equation}
It is routine to check that
\begin{equation}\label{SnE-GF}
C^E(x,z)=\frac{x-e^{(x-1)z}}{x-e^{2(x-1)z}},~C^O(x,z)=\frac{1-e^{(x-1)z}}{e^{2(x-1)z}-x}.
\end{equation}

It is well known that the exponential generating functions of $A_n(x)$ and $B_n(x)$ are given as follows (see~\cite[Theorem~3.4]{Brenti94}):
\begin{equation*}\label{An-GF}
A(x,z):=\sum_{n=0}^\infty A_n(x)\frac{z^n}{n!}=\frac{x-1}{x-e^{(x-1)z}}=\frac{(1-x)e^{(1-x)z}}{1-xe^{(1-x)z}},
\end{equation*}
\begin{equation*}\label{An-GF}
B(x,z):=\sum_{n=0}^\infty B_n(x)\frac{z^n}{n!}=\frac{(1-x)e^{(1-x)z}}{1-xe^{2(1-x)z}}.
\end{equation*}

We can now present the third main result of this paper.
\begin{theorem}\label{mainthm03}
We have
$$\frac{A(x,2z)}{A(x,z)}=C^E(x,z),~\frac{B(x,z)}{A(x,z)}=1+xC^O(x,z).$$
\end{theorem}
\begin{proof}
Note that
$$A(x,2z)=\frac{x-1}{x-e^{2(x-1)z}}=\frac{x-e^{(x-1)z}}{x-e^{2(x-1)z}}\frac{x-1}{x-e^{(x-1)z}}=C^E(x,z)A(x,z),$$
which yields the first formula. Since
\begin{equation*}
1+xC^O(x,z)=\frac{1-xe^{(1-x)z}}{1-xe^{2(1-x)z}},
\end{equation*}
it follows that
$$B(x,z)=\frac{(1-x)e^{(1-x)z}}{1-xe^{2(1-x)z}}=\frac{1-xe^{(1-x)z}}{1-xe^{2(1-x)z}}\frac{(1-x)e^{(1-x)z}}{1-xe^{(1-x)z}}=\left(1+xC^O(x,z)\right)A(x,z).$$
This completes the proof.
\end{proof}
An immediate consequence of Theorem~\ref{mainthm03} is the following corollary.
\begin{corollary}
For $n\geq 0$, we have
$$2^nA_n(x)=\sum_{k=0}^n\binom{n}{k}A_k(x)C_{n-k}^E(x),$$
$$B_n(x)=A_n(x)+x\sum_{k=0}^{n-1}\binom{n}{k}A_k(x)C_{n-k}^O(x).$$
\end{corollary}

In~\cite{Bagno04}, Bagno and Garber introduced the definition of flag excedance of signed permutations.
Let $\pi\in\mbn$.
The {\it flag excedance number} of $\pi$ is defined by
$$\fexc(\pi):=2\#\{i\in [n]:~\pi(i)>i\}+\#\{i\in [n]:~\pi(i)<0\}.$$
Let $\exc_A(\pi)=\#\{i\in [n]:~\pi(i)>i\}$.
Then $\fexc(\pi)=2\exc_A(\pi)+\negg(\pi)$.
It is well known that
\begin{equation}\label{fdesfexc}
\sum_{\pi\in\mbn}x^{\fdes(\pi)}=\sum_{\pi\in\mbn}x^{\fexc(\pi)},
\end{equation}
which has been extended to colored permutations and affine Weyl groups.
Ordinary and $q$-generalizations of~\eqref{fdesfexc} have been pursued by several authors, see~\cite{Foata11,Mongelli13,Zeng16} for instance.
For example, Mongelli~\cite{Mongelli13} derived some basic properties of the flag excedance polynomials of classical and affine Weyl groups. In particular,
Mongelli~\cite[p.~1221-1222]{Mongelli13} presented a combinatorial proof of~\eqref{fdesfexc}. Recall that $\mdn$ is the Coxeter group of type $D$.
According to~\cite[Corollary~5.1]{Mongelli13},
\begin{equation}\label{Dnfexc}
\sum_{\pi\in\mdn}x^{\fexc(\pi)}=\frac{1}{2}\left((1+x)^nA_n(x)+(1-x)^nA_n(-x)\right).
\end{equation}

By~\eqref{fdesfexc}, we see that $$C_n(x)=\sum_{\pi\in\mdn}x^{\fexc(\pi)}+\sum_{\pi\in\mbn\setminus\mdn}x^{\fexc(\pi)}.$$
Let $C(x,z):=\sum_{n=0}^\infty C_n(x)\frac{z^n}{n!}$. It follows from~\eqref{fdesAnx} that $C(x,z)=A(x,(1+x)z)$.
By using~\eqref{Dnfexc}, we immediately get
\begin{equation*}
\sum_{n=0}^\infty \sum_{\pi\in\mdn}x^{\fexc(\pi)}\frac{z^n}{n!}=\frac{1}{2}\left(C(x,z)+C(-x,z)\right)=C^E(x^2,z).
\end{equation*}
Thus
\begin{equation*}
\sum_{n=0}^\infty \sum_{\pi\in\mbn\setminus \mdn}x^{\fexc(\pi)}\frac{z^n}{n!}=xC^O(x^2,z).
\end{equation*}
So we get the following result, which gives a connection between Coxeter groups $\mbn$ and $\mdn$.
\begin{proposition}\label{CnxBnxdes}
Let $C_n^+=\{\pi\in\mbn:~\pi(1)>0\}$. Then we have
\begin{equation*}
\sum_{\pi\in\mdn}x^{\fexc(\pi)}=\sum_{\pi\in C_n^+}x^{2\des_A(\pi)}.
\end{equation*}
\end{proposition}
A combinatorial proof of Proposition~\ref{CnxBnxdes} would be interesting.
\subsection{Flag ascent-plateau polynomials for Stirling permutations}
\hspace*{\parindent}

Stirling permutations were introduced by Gessel and Stanley~\cite{Gessel78}.
A {\it Stirling permutation} of order $n$ is a permutation of the multiset $\{1,1,2,2,\ldots,n,n\}$ such that
for each $i$, $1\leq i\leq n$, all entries between the two occurrences of $i$ are larger than $i$.
The reader is referred to~\cite{Bona08,Haglund12,Ma17,Ma2001} for the recent study on Stirling permutations.

Denote by $\mqn$ the set of {\it Stirling permutations} of order $n$.
Let $\sigma=\sigma_1\sigma_2\cdots\sigma_{2n}\in\mqn$.
The {\it ascent-plateau number}, {\it left ascent-plateau number} and {\it flag ascent-plateau number} of $\sigma$ are respectively defined by
\begin{align*}
\ap(\pi)&=\#\{i\in[2,2n-1]:~\pi(i-1)<\pi(i)=\pi(i+1)\},\\
\lap(\pi)&=\#\{i\in[2n-1]:~\pi(i-1)<\pi(i)=\pi(i+1),~\pi(0)=0\},\\
\fap(\pi)&=\left\{
               \begin{array}{ll}
                 2\ap(\sigma)+1, & \hbox{if $\sigma_1=\sigma_2$;} \\
                 2\ap(\sigma), & \hbox{otherwise.}
               \end{array}
             \right.
\end{align*}
Clearly, $\fap(\sigma)=\ap(\sigma)+\lap(\sigma)$.
The {\it flag ascent-plateau polynomials} $L_n(x)$ are defined by $$L_n(x)=\sum_{\sigma\in\mqn}x^{\fap(\sigma)}.$$
The polynomials $L_n(x)$ satisfy the recurrence relation
\begin{equation}\label{Fnx-recu}
L_{n+1}(x)=(x+2nx^2)L_n(x)+x(1-x^2)L_n'(x),
\end{equation}
with $L_0(x)=1$ (see~\cite[p.~14]{Ma2001}). The first few $L_n(x)$ are
\begin{align*}
L_1(x)&=x,\\
L_2(x)&=x+x^2+x^3,\\
L_3(x)&=x+3x^2+7x^3+3x^4+x^5.
\end{align*}
Set $\mqn=\mqn^+\cup\mqn^-$, where $\mqn^+=\{\sigma\in\mqn:~\sigma_1<\sigma_2\}$ and $\mqn^-=\{\sigma\in\mqn:~\sigma_1=\sigma_2\}$.
Then
\begin{equation}\label{Lnx-express}
L_n(x)=L_n^E(x^2)+xL_n^O(x^2)=\sum_{\sigma\in\mqn^+}x^{2\ap(\sigma)}+x\sum_{\sigma\in\mqn^-}x^{2\ap(\sigma)}.
\end{equation}

From~\eqref{Fnx-recu}, we see that $\alpha_n(x)=x+2nx^2,~\beta_n(x)=x(1-x^2)$ and $\beta_n^e(x)=0$.
Using~\eqref{dualrelation01}, we obtain $x(x+2nx^2)=xp_n(x^2)+r_n(x^2)$ and $2x^2(1-x^2)=q_n(x^2)$.
Hence $$p_n(x)=2nx,~q_n(x)=2x(1-x),~r_n(x)=x,$$
$$u_n(x)=2nx+1-x=(2n-1)x+1,~v_n(x)=2x(1-x),~w_n(x)=1.$$
Therefore, by Theorem~\ref{mathm01}, we can now present the fouth main result of this paper.
\begin{theorem}
For $n\geq1$, we have
\begin{equation}\label{Tnx-recu}
\left\{
  \begin{array}{l}
    L_{n+1}^E(x)=2nxL_n^E(x)+2x(1-x)\frac{d}{dx}{L_n^E}(x)+xL_n^O(x), \\
    L_{n+1}^O(x)=((2n-1)x+1)L_n^O(x)+2x(1-x)\frac{d}{dx}L_n^O(x)+L_n^E(x),
  \end{array}
\right.
\end{equation}
with the initial conditions $L_1^E(x)=0$ and $L_1^O(x)=1$.
\end{theorem}

The {\it ascent-plateau polynomials} and {\it left ascent-plateau polynomials} are defined by
$$M_n(x)=\sum_{\sigma\in\mqn}x^{\ap(\pi)},~N_n(x)=\sum_{\sigma\in\mqn}x^{\lap(\pi)}.$$
According to~\cite[Theorem 2,~Theorem 3]{Ma15}, we have
\begin{equation*}\label{Mxt}
M(x,t)=\sum_{n\geq 0}M_n(x)\frac{t^n}{n!}=\sqrt{\frac{x-1}{x-e^{2t(x-1)}}},
\end{equation*}
\begin{equation*}\label{Nxt}
N(x,t)=\sum_{n\geq 0}N_n(x)\frac{t^n}{n!}=\sqrt{\frac{1-x}{1-xe^{2t(1-x)}}}.
\end{equation*}
Clearly, $M_n(x)=x^nN_n\left(\frac{1}{x}\right)$. Moreover, according to~\cite[Proposition~1]{Ma17}, we have
\begin{align*}
2^nxA_n(x)&=\sum_{i=0}^n\binom{n}{i}N_i(x)N_{n-i}(x),\\
B_n(x)&=\sum_{i=0}^n\binom{n}{i}N_i(x)M_{n-i}(x).
\end{align*}
It is well known that $A_n(x)$ and $B_n(x)$ are both unimodal (see~\cite{Athanasiadis17} for instance).
Motivated by the above convolution formulas, it is natural to explore the unimodality of $M_n(x)$ and $N_n(x)$.

Let $f(x)=\sum_{i=0}^nf_ix^i$.
We say that $f(x)$ is {\it unimodal} if there exists an index $0\leq k\leq n$ such that $f_0\leq f_1\leq \cdots\leq f_k\geq f_{k+1}\geq\cdots \geq f_n$.
Such an index $k$ is called a {\it mode} of $f(x)$. In the past decades, it is a fundamental problem to determine the location of modes of a sequence
of combinatorial polynomials, see~\cite{Chen09,Wang2005} for instance.
Following~\cite[Definition~2.9]{Schepers13},
we say that $f(x)$ is {\it alternatingly increasing} if
$$f_0\leq f_n\leq f_1\leq f_{n-1}\leq\cdots f_{\lrf{\frac{n+1}{2}}}.$$
Clearly, alternatingly increasing property is a stronger property than unimodality. Very recently,
there has been much work on the alternatingly increasing property of Ehrhart polynomials (see~\cite{Beck2019,Branden18,Solus19}).
In the rest part of this subsection, we shall show that
both $M_n(x)$ and $N_n(x)$ are alternatingly increasing.

If $f(x)$ is symmetric and $\deg f(x)=n$, then it can be expanded as
$$f(x)=\sum_{k=0}^{\lrf{{n}/{2}}}\gamma_kx^k(1+x)^{n-2k},$$ and it is said to be {\it $\gamma$-positive}
if $\gamma_k\geq 0$ for $0\leq k\leq \lrf{\frac{n}{2}}$.
If $f(x)$ is $\gamma$-positive, then $f(x)$ is unomidal and symmetric.
Following~\cite[Definition~15]{Ma2001},
if $$f(x)=(1+x)^\nu\sum_{k=0}^{n}\lambda_kx^k(1+x^2)^{n-k}$$ and $\lambda_k\geq 0$ for all $0\leq k\leq n$,
then we say that $f(x)$ is {\it semi-$\gamma$-positive}, where $\nu=0$ or $\nu=1$.
Therefore, if $f(x)$ is semi-$\gamma$-positive, then
\begin{align*}
f(x)&=(1+x)^\nu\left(\sum_{k=0}^{\lrf{n/2}}\lambda_{2k}x^{2k}(1+x^2)^{n-2k}+\sum_{k=0}^{\lrf{(n-1)/2}}\lambda_{2k+1}x^{2k+1}(1+x^2)^{n-1-2k}\right)\\
&=(1+x)^\nu(f_1(x^2)+xf_2(x^2)),
\end{align*}
and both $f_1(x)$ and $f_2(x)$ are $\gamma$-positive (see~\cite[Proposition~16]{Ma2001}).

We now recall a very recent result on the flag ascent-plateau polynomials.
\begin{proposition}[{\cite[Theorem~19]{Ma2001}}]\label{Lnx-prop}
The polynomial $L_n(x)$ is semi-$\gamma$-positive. More precisely, for $n\geq 1$,
we have $$L_n(x)=\sum_{k=1}^nL_{n,k}x^k(1+x^2)^{n-k},$$
where the numbers $L_{n,k}$ satisfy the recurrence relation
\begin{equation}\label{Lnk-recu}
L_{n+1,k}=kL_{n,k}+L_{n,k-1}+4(n-k+2)L_{n,k-2},
\end{equation}
with the initial conditions $L_{0,0}=1$ and $L_{n,0}=0$ for $n\geq 1$.
\end{proposition}

It follows from~\eqref{Lnx-express} that
\begin{equation}\label{MnxNnx}
M_n(x)=L_n^E(x)+L_n^O(x),~N_n(x)=L_n^E(x)+xL_n^O(x).
\end{equation}
Since $\deg L_n(x)=2n-1$, we have $\deg L_n^E(x)=\deg L_n^O(x)=n-1$.
Note that $L_n^E(0)=0$ and $L_n^O(0)=1$.
By using Proposition~\ref{Lnx-prop}, we get that both $L_n^E(x)$ and $L_n^O(x)$ are $\gamma$-positive for any $n\geq 1$.
More precisely, we have
\begin{equation}\label{LnELnO}
\left\{
  \begin{array}{l}
L_n^E(x)=\sum_{\sigma\in\mqn^+}x^{\ap(\sigma)}=\sum_{k=1}^{\lrf{n/2}} L_{n,2k} x^k(1+x)^{n-2k},, \\
L_n^O(x)=\sum_{\sigma\in\mqn^-}x^{\ap(\sigma)}=\sum_{k=0}^{\lrf{(n-1)/2}} L_{n,2k+1}x^k(1+x)^{n-1-2k}.
  \end{array}
\right.
\end{equation}
%
In conclusion, we now present the fifth main result of this paper.
\begin{theorem}
For any $n\geq 1$, both $M_n(x)$ and $N_n(x)$ are alternatingly increasing, and so they are unimodal with modes in the middle.
\end{theorem}
\begin{proof}
It follows from~\eqref{LnELnO} that both $L_n^E(x)$ and $L_n^O(x)$ are symmetric and unimodal. When $n=2m+1$,
assume that
\begin{align*}
L_{2m+1}^E(x)&=\ell_1x+\ell_2x^2+\cdots+\ell_{m-1}x^{m-1}+\ell_mx^m+\ell_{m}x^{m+1}+\ell_{m-1}x^{m+2}+\cdots \ell_2x^{2m-1}+\ell_1x^{2m},\\
L_{2m+1}^O(x)&=1+\widetilde{\ell}_1x+\widetilde{\ell}_2x^2+\cdots+\widetilde{\ell}_{m-1}x^{m-1}+\widetilde{\ell}_mx^m+\widetilde{\ell}_{m-1}x^{m+1}+\cdots+\widetilde{\ell}_1x^{2m-1}+x^{2m}.
\end{align*}
Then $M_{2m+1}(x)=L_{2m+1}^E(x)+L_{2m+1}^O(x)=\sum_{i=0}^{2m}M_{2m+1,i}x^i$,
where
$$M_{2m+1,i}=\left\{
  \begin{array}{ll}
    1, & \hbox{if $i=0$;} \\
    \ell_i+\widetilde{\ell}_i, & \hbox{if $1\leq i\leq m$;} \\
    \ell_{2m-i+1}+\widetilde{\ell}_{2m-i}, & \hbox{if $m+1\leq i\leq 2m-1$;} \\
    \ell_1+1, & \hbox{if $i=2m$.}
  \end{array}
\right.$$
It is clear that $1\leq \ell_1+1\leq \ell_1+\widetilde{\ell}_1\leq \ell_2+\widetilde{\ell}_1\leq \cdots\leq \ell_m+\widetilde{\ell}_m$, i.e.,
$$M_{2m+1,0}\leq M_{2m+1,2m}\leq M_{2m+1,1}\leq M_{2m+1,2m-1}\leq \cdots\leq M_{2m+1,m}.$$
Note that $N_{2m+1}(x)=L_{2m+1}^E(x)+xL_{2m+1}^O(x)=\sum_{i=0}^{2m+1}N_{2m+1,i}x^i$,
where
$$N_{2m+1,i}=\left\{
  \begin{array}{ll}
    0, & \hbox{if $i=0$;} \\
    \ell_1+1, & \hbox{if $i=1$;} \\
    \ell_i+\widetilde{\ell}_{i-1}, & \hbox{if $2\leq i\leq m$;} \\
    \ell_{2m-i+1}+\widetilde{\ell}_{2m-i+1}, & \hbox{if $m+1\leq i\leq 2m$;} \\
    1, & \hbox{if $i=2m+1$.}
  \end{array}
\right.$$
It is clear that $0<1\leq \ell_1+1\leq \ell_1+\widetilde{\ell}_1\leq \ell_2+\widetilde{\ell}_1\leq \cdots\leq \ell_m+\widetilde{\ell}_m$, i.e.,
$$N_{2m+1,0}\leq N_{2m+1,2m+1}\leq N_{2m+1,1}\leq N_{2m+1,2m}\leq \cdots\leq N_{2m+1,m+1}.$$
Therefore, both $M_{2m+1}(x)$ and $N_{2m+1}(x)$ are alternatingly increasing, the mode of $M_{2m+1}(x)$ is $m$ and that of $N_{2m+1}(x)$ is $m+1$.
In the same way, one can verify that
both $M_{2m}(x)$ and $N_{2m}(x)$ are alternatingly increasing with the mode $m$. This completes the proof.
\end{proof}
\subsection{The up-down run polynomials for symmetric group}
\hspace*{\parindent}

Let $\pi\in\msn$. The {\it up-down runs} of a permutation $\pi\in\msn$ are the alternating runs of $\pi$ endowed with a 0
in the front~(see~\cite{Taylor03,Ma2001}). Let $\udrun(\pi)$ denote the number of up-down runs of $\pi$.
The {\it interior peak} number and {\it left peak} number of $\pi$ are respectively defined by
\begin{align*}
\ipk(\pi)&=\#\{i\in[2,n-1]:~\pi(i-1)<\pi(i)>\pi(i+1)\},\\
\lpk(\pi)&=\#\{i\in[n-1]: \pi(i-1)<\pi(i)>\pi(i+1),~\pi(0)=0\}.
\end{align*}
Define
$$W_n(x)=\sum_{\pi\in\msn}x^{\ipk(\pi)},~\overline{W}_n(x)=\sum_{\pi\in\msn}x^{\lpk(\pi)}.$$
The polynomials $W_n(x)$ and $\overline{W}_n(x)$ satisfy the recurrence rleations
$$W_{n+1}(x)=(nx-x+2)W_n(x)+2x(1-x)W_n'(x),$$
$$\overline{W}_{n+1}(x)=(nx+1)\overline{W}_n(x)+2x(1-x)\overline{W}_n'(x),$$
with $W_1(x)=\overline{W}_1(x)=1$ (see~\cite{Ma121,Petersen06}).
Note that $\deg \overline{W}_n(x)\geq \deg W_n(x)$. Then we
set $$p_n(x)=nx+1,~q_n(x)=2x(1-x),~r_n(x)=0.$$ Note that
$p_n(x)+\frac{1}{2x}q_n(x)=nx-x+2$. Then by using Theorem~\ref{mathm01}, we can define
\begin{align*}
\alpha_n(x)&=p_n(x^2)+\frac{1}{x}r_n(x^2)=nx^2+1,~
\beta_n(x)=\frac{1}{2x}q_n(x^2)=x(1-x^2).
\end{align*}
Therefore, we immediately get the following result, which has been proved in~\cite{Ma121}.
\begin{proposition}
Let $\{R_n(x)\}_{n\geq 1}$ be a sequence of polynomials defined by
\begin{equation}\label{Rnx01}
R_{n+1}(x)=(nx^2+1)R_n(x)+x(1-x^2)R_n'(x),
\end{equation}
with $R_1(x)=1+x$. Then $R_n(x)=\overline{W}_n(x^2)+xW_n(x^2)$.
\end{proposition}

The {\it up-down run polynomials} $T_n(x)$ are defined by $$T_n(x)=\sum_{\pi\in\msn}x^{\udrun(\pi)}.$$
The polynomials $T_n(x)$ satisfy the recurrence relation
\begin{equation}\label{Rnx02}
T_{n+1}(x)=x(nx+1)T_{n}(x)+x\left(1-x^2\right)T_{n}'(x),
\end{equation}
with initial conditions $T_0(x)=1$ and $T_1(x)=x$ (see~\cite{Ma132,Stanley08}).
Comparing~\eqref{Rnx01} with~\eqref{Rnx02}, it is easy to check that $$R_n(x)=\frac{1+x}{x}T_{n}(x).$$

Set $\msn^+=\{\pi\in\msn:~\pi(n-1)>\pi(n)\}$ and $\msn^-=\{\pi\in\msn:~\pi(n-1)<\pi(n)\}$.
We define
$$T_n^E(x)=\sum_{\pi\in\msn^+}x^{\lpk(\pi)},~T_n^O(x)=\sum_{\pi\in\msn^-}x^{\lpk(\pi)}.$$
By definition, we get the following result.
\begin{proposition}\label{PropUnx}
For $n\geq 1$, we have
\begin{equation*}
\left\{
  \begin{array}{l}
T_n(x)=T_n^E(x^2)+xT_n^O(x^2),\\
\overline{W}_n(x)=T_n^E(x)+T_n^O(x).
  \end{array}
\right.
\end{equation*}
\end{proposition}

From~\eqref{Rnx02}, we see that $\alpha_n(x)=nx^2+x,~\beta_n(x)=x(1-x^2)$ and $\beta_n^e(x)=0$.
Using~\eqref{dualrelation01}, we obtain $x(nx^2+x)=xp_n(x^2)+r_n(x^2)$ and $2x^2(1-x^2)=q_n(x^2)$.
Hence $$p_n(x)=nx,~q_n(x)=2x(1-x),~r_n(x)=x,$$
$$u_n(x)=nx+1-x=(n-1)x+1,~v_n(x)=2x(1-x),~w_n(x)=1.$$

By using~\eqref{Rnx02} and~\cite[Theorem~2]{Ma2008}, we see that $T_n(x)$ have only real nonpositive zeros.
Combining Theorem~\ref{mathm01} and Theorem~\ref{Hermite}, we obtain the sixth main result of this paper.
\begin{theorem}
For $n\geq1$, we have
\begin{equation}\label{Tnx-recu}
\left\{
  \begin{array}{l}
    T_{n+1}^E(x)=nxT_n^E(x)+2x(1-x)\frac{d}{dx}{T_n^E}(x)+xT_n^O(x), \\
    T_{n+1}^O(x)=((n-1)x+1)T_n^O(x)+2x(1-x)\frac{d}{dx}T_n^O(x)+T_n^E(x),
  \end{array}
\right.
\end{equation}
with the initial conditions $T_1^E(x)=0$ and $T_1^O(x)=1$.  Moreover, both $T_n^E(x)$ and $T_n^O(x)$ have only real negative zeros and $T_n^O(x)\prec T_n^E(x)$.
\end{theorem}
\subsection{Alternating run polynomials for signed permutations}
\hspace*{\parindent}

Let $\pi\in\mbn$. The peak number and valley number of $\pi$ are respectively defined by
\begin{align*}
\pk(\pi)&=\#\{i\in[n-1]:~\pi(i-1)<\pi(i)>\pi(i+1),~\pi(0)=0\},\\
\val(\pi)&=\#\{i\in[n-1]: \pi(i-1)>\pi(i)<\pi(i+1),~\pi(0)=0\}.
\end{align*}
Recall that $C_n^+=\{\pi\in\mbn:~\pi(1)>0\}$.
We define
\begin{equation}\label{PnxVnx-def}
\begin{split}
U_n(x)&=\sum_{\pi\in C_n^+}x^{\pk(\pi)},~
V_n(x)=\sum_{\pi\in C_n^+}x^{\val(\pi)}.
\end{split}
\end{equation}

According to~\cite[Corollary~7]{Chow14}, the polynomials $U_n(x)$ and $V_n(x)$ satisfy the following system of Eulerian recurrences:
\begin{equation*}
\left\{
  \begin{array}{l}
U_{n+1}(x)=(2nx+1)U_{n}(x)+4x(1-x)U_{n}'(x)+xV_{n}(x),\\
V_{n+1}(x)=(2nx-2x+3)V_{n}(x)+4x(1-x)V_{n}'(x)+U_{n}(x),
  \end{array}
\right.
\end{equation*}
with $U_{0}(x)=1$ and $V_{0}(x)=0$. Note that $\deg U_n(x)\geq \deg V_n(x)$. Thus $(U_n(x),V_n(x))$ is a Eulerian pair.
Put
$$p_n(x)=2nx+1,~q_n(x)=4x(1-x),~r_n(x)=x,$$
$$u_n(x)=2nx-2x+3,~v_n(x)=4x(1-x),~w_n(x)=1.$$
Then $\alpha_n(x)=p_n(x^2)+\frac{1}{x}r_n(x^2)=2nx^2+x+1$ and
$\beta_n(x)=\frac{1}{2x}q_n(x^2)=2x(1-x^2)$.
Then by using Theorem~\ref{mathm01}, we immediately get the following result.
\begin{proposition}
Let $\{H_n(x)\}_{n\geq 0}$ be a sequence of polynomials defined by
\begin{equation}\label{Rnx01}
H_{n+1}(x)=(2nx^2+x+1)H_n(x)+2x(1-x^2)H_n'(x),
\end{equation}
with $H_0(x)=1$. Then $H_n(x)=U_n(x^2)+xV_n(x^2)$.
\end{proposition}

Let $\altrun(\pi)=\pk(\pi)+\val(\pi)$ be the number of alternating runs of $\pi$.
The alternating run polynomials for signed permutations are given as follows:
$$\widetilde{H}_n(x)=\sum_{\pi\in C_n^+}x^{\altrun(\pi)}.$$
Zhao~\cite{Zhao11} showed that
$\widetilde{H}_{n+1}(x)=(2nx^2+3x-1)\widetilde{H}_{n}(x)+2x\left(1-x^2\right)\widetilde{H}_{n}'(x)$ for $n\geqslant 1$,
with $\widetilde{H}_1(x)=x$. It is routine to verify that $$H_n(x)=\frac{1+x}{x}\widetilde{H}_n(x)$$ for $n\geq 1$, which has been proved in~\cite[Theorem~8]{Chow14}.
%
%
\section{Concluding remark}
In this paper, we consider the combinatorial aspects of the Hermite-Biehler decompositions of several enumerative polynomials.
Let $\{f_n(x)\}_{n\geq 0}$ be a sequence of polynomials with nonnegative coefficients. Suppose that
\begin{equation}\label{recufnx}
f_{n+1}(x)=\left(a_1n+a_2+(b_1n+b_2)x+(c_1n+c_2)x^2\right)f_n(x)+dx(1-x^2)f_n'(x).
\end{equation}
where $a_1,a_2,b_1,b_2,c_1,c_2,d\in\mathbb{R}$. Then
$$\alpha_n(x)=a_1n+a_2+(b_1n+b_2)x+(c_1n+c_2)x^2,~\beta_n(x)=dx(1-x^2).$$
It follows from~\eqref{dualrelation01} that
$$p_n(x)=a_1n+a_2+(c_1n+c_2)x,~q_n(x)=2dx(1-x),~r_n(x)=(b_1n+b_2)x.$$
By using Theorem~\ref{mathm01}, we obtain
$$u_n(x)=a_1n+a_2+d+(c_1n+c_2-d)x,~v_n(x)=2dx(1-x),~w_n(x)=b_1n+b_2,$$
and then we can derive the recurrence system of the polynomials $f^E(x)$ and $f^O(x)$.

Besides the polynomials discussed in Section~\ref{Section02}, many other enumerative polynomials also satisfy the recurrence~\eqref{recufnx}, see~\cite{Barbero14,Zhu,Zhuang17} for instance.
We end this paper by giving an example. Following~\cite[Definition~1]{Aval13}, a tree-like tableau is a Ferrers diagram where each cell
contains either 0 or 1 point with some constraints. The symmetric tableaux are tree-like tableaux which are invariant with respect to reflection
through the main diagonal of their diagram.
Let $b(n,k)$ be the number of symmetric tableaux of size $2n+1$ with $k$ diagonal cells, and let $b_n(x)=\sum_{k=1}^{n+1}b(n,k)x^k$.
It follows from~\cite[Proposition~18]{Aval13} that $$b_{n+1}(x)=(n+1)x(1+x)b_n(x)+x(1-x^2)b_n'(x),$$
with the initial condition $b_0(x)=x$. By using the recurrence system of the polynomials $b_n^E(x)$ and $b_n^O(x)$,
one can easily derive that $b_n^E(1)=b_n^O(1)=2^{n-1}n!$ for $n\geq 1$. We leave the details to the reader.
Thus it may be interesting to further explore properties of $b_n^E(x)$ and $b_n^O(x)$.
\section*{Acknowledgements.}
This work is supported by NSFC (12071063,11571235) and
NSC (108-2115-M-017-005-MY2,~107-2115-M-001-009-MY3).

\end{document}